\documentclass[12pt,a4paper]{article}
\usepackage{amsmath}
\usepackage{amsfonts}
\usepackage{amssymb}
\usepackage{natbib}
\usepackage{algorithm}% http://ctan.org/pkg/algorithms
\usepackage{algpseudocode}% http://ctan.org/pkg/algorithmicx
\usepackage{amsthm}
\usepackage{dsfont}

\usepackage{graphicx}
\usepackage{mathtools}
 
\newtheorem{theorem}{Theorem}[section]

\newtheorem{lemma}[theorem]{Lemma}
\newtheorem{definition}{Definition}

\numberwithin{equation}{section}

\begin{document}

\begin{titlepage}
   \vspace*{\stretch{1.0}}
   \begin{center}
      \Large\textbf{Finding Many Sparse Cuts Using Entropy Maximization}\\
      
      \Large\textbf{(March 8th, 2022)}  \\ 
     
	  \large\textit{Farshad Noravesh}\footnote{Email: noraveshfarshad@gmail.com}

   \end{center}
   \vspace*{\stretch{2.0}}
\begin{abstract}
A randomized algorithm for finding sparse cuts is given which is based on constructing a dual markov chain called multiscale rings process(MRP) and a new concept of entropy. It is shown how the time to absorption of the dual process measures the connectedness of the graph and mixing of the corresponding markov process which is then utilized to do clustering. The second algorithm uses the entropy which provides a new methodology and a set of tools to think about sparse cuts as well as sparsification of a graph. 
\end{abstract}

\end{titlepage}

\section{Introduction}
Partitiong a graph so that each cluster is densely connected with high expansion and clusters are sparsely connected to each other has always been an important problem in several areas such as computer science, big data, machine learning and bioinformatics. There are three major paradigms to deal with such a problem. The first approach comes from approximation algorithms and doing a combinatorial optimization and LP/SDP relaxations and roundings are the main classical tools specially in computer science.
The second paradigm arises from the desire to generalize cheeger inequality  and there are many important attempts such as \citep{Louis2012},\citep{Lee2014}.
The last paradigm which is relatively less appreciated is converting the problem to another problem that has the goal of finding all possible ways to reduce mixing time within each cluster. The absence of bottlenecks in the state space of a Markov chain would imply rapid mixing. Research on this paradigm is concentrated on some concepts as a proxy to mixing time. One approach is to use sparsification and electrical networks \citep{Spielman2008} and in the intrinsic level it uses hitting times. Another approach is dealing with mixing time of a markov chain \citep{Andersen2009},\citep{Morris2005}. A very important article for showing how different concepts like mixing times, hitting times are related to each other is discovered in \citep{Peres2011}.

In the present paper the shortest path metric on a graph is denoted by $d_{G}$.\\
The $\epsilon-$mixing time is defined as
\begin{equation}\label{eq-epsilonmixingtime}
\tau(\epsilon):=min \{n:||p^{n}(x,.)-\pi||\leq \epsilon \ for \ all \  x\in V \}
\end{equation}
\citep{Morris2005} is motivated by \citep{Lovasz1999} who proved that 
\begin{equation}
\tau(1/4)\leq 2000 \int_{\pi_*}^{3/4}\frac{du}{u\Phi^2(u)}
\end{equation}

There are many approaches to bounding the mixing time. One of them is using relaxation time as a proxy for mixing time. In fact markov chains can be characterized by some important concepts such as hitting time,cover time,relaxation time and mixing time and they are all related to each other. For example \citep{Oliveira2019} proved an upper bound for hitting time:
\begin{equation}\label{eq-Oliveira-hitting-bound}
t_{hit}\leq \frac{20d_{avg}}{d_{min}}n\sqrt{t_{rel} +1}
\end{equation} 

Another approach is to use coupling and there are many examples such as graph coloring and finding independent sets based on the idea of coupling and more specifically path coupling.
One of the most natural ideas to bound mixing time goes back to idea of \citep{Aldous1987} who defined strong stationary times and later \citep{Diaconis1990} constructed a dual markov chain and analyzed the time to absorption of the dual process and later \citep{Morris2005} developed this idea and then used it for finding sparse cuts locally and naming this dual process the evolving set process(ESP). \citep{Andersen2009} further developed the idea to construct a non-vanishing process called volume biased evolving set process(VB-ESP). \citep{Andersen2009} proved that the cardinality of evolving set process is a martingale.

 It can be easily observed that the probability of hitting 2k within 2n time steps for simple random walk is
\begin{equation}\label{eq-hitting-RW}
\mathbb{P}(S_{2n}=2k|S_{0}=0)={2n \choose n+k}p^{n+k}q^{n-k}, \ -n\leq k \leq n
\end{equation}

Hoeffding's theorem states that, for all $k > 0$
\begin{equation}\label{eq-Hoeffding}
\mathbb{P}(S_{n}-E(S_n)\geq k)\leq exp(-\frac{2k^2}{\sum_{i=1}^n (b_i-a_i)^2})
\end{equation}

The probability that the random walk on $\mathbb{Z}$ is at k after nth step is a direct consequence of Hoeffding's inequality
\begin{equation}\label{eq-hit-afternstep}
\sum_{|k|\geq d}q_{n}(k)\leq 2e^{-d^2/(2n)}
\end{equation}

\citep{Carne1985} proved the following theorem

\begin{theorem}\label{thm:carne-var} 
Denote by P the $L^2(\mu)$ operator associated to the transition kernel p and let $|P|$ stand for its norm which is always less than or equal to one. Then:

\begin{equation}\label{eq-carne85}
p^{t}(x,y)\leq 2(\frac{\mu(y)}{\mu(x)})^{1/2} \ |P|^{t} \  exp(\frac{-d(x,y)^2}{2t})
\end{equation}

\end{theorem}

\section{Construction of MRP}
Subordination is an important tool to classify and analyze Markov processes such as brownian motion\citep{Kim2020}  which is also extended to the discrete case in \citep{Bendikov2012} and many different variations of it are defined in the literature for different purposes such as\citep{Ante2015}  . Here the following definition is used in the present paper to define MRP at any scale.
\begin{definition}
A subordinated random walk is defined as
\begin{equation} \label{subordinated}
Y_{n}:=S_{\tau_{n}}
\end{equation}
where $S_n=x+X_1+X_2+\ldots+X_n$ is a simple random walk starting at $x\in Z^d$ and $\tau_n=R_1+R_2+\ldots+R_n$ is a random ealk on $Z_{+}$ with independent increments $R_i$ having a probability distribution.
\end{definition}
\begin{definition}
Multiscale Rings Process(MRP) at scale s is a subordinated random walk defined in \ref{subordinated} while $\tau_n$ is stopped at a fixed threshold time $T_{th}^{(s)}$. The scale s is induced by the distribution of $R_i$(increments on $Z_+$)
\end{definition}

\begin{definition}
A ring at position r of scale s is defined by the set of points inside the following zone
\begin{equation}\label{eq-ring} 
ring_{s}(r)=\{x: r\leq  x\leq r+2^s  \}  
\end{equation}
\end{definition}

\begin{figure}[H]
  \includegraphics[width=\linewidth]{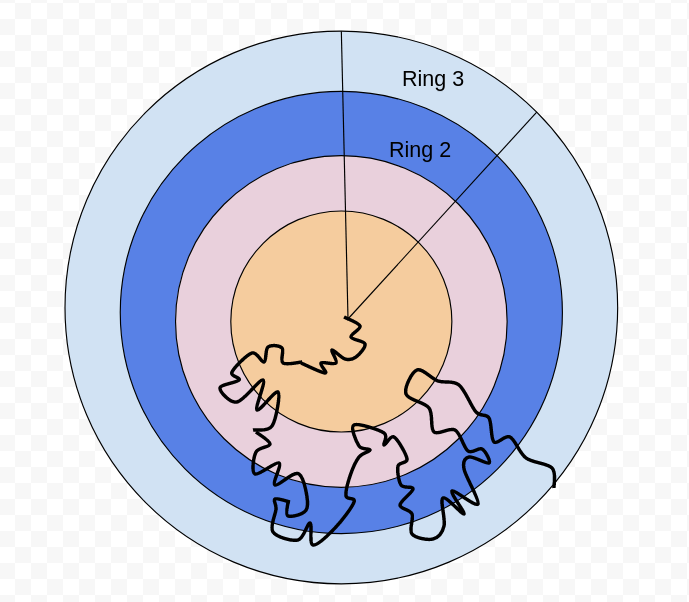}
  \caption{a random walk starting from the center and the rings $\Omega_i$}
  \label{fig-ring}
\end{figure}

\begin{definition}
Solar system associated to a center node is the union of all rings up to a point whose shortest path distance to center node is bounded by $10\times 2^{s}$ where s is the fixed length scale of all rings at scale s.
\end{definition}
A solar system is depicted in Figure~\ref{fig-ring}.
The algorithm~\ref{alg:clustering} starts by creating a metric embedding into $\ell^{2}$ using any favorite algorithm. Note that there may be a tradeoff between distortion of metric embedding and the time complexity of creating such an algorithm.
After creating the metric embedding and representing it as a polar transformation, the solar system associated to a center is created by ten rings.
\begin{definition}
Any ring is divided into equal number of blocks where each block has only two negihbor blocks in that ring and has two other neighbors from previous and next rings. $Ring(r)=\bigcup_{i=1}^{n_b}B^{(r)}_{i}$
\end{definition}
Note that there is no constraint on the graph to be planar and the blocks of a ring can be connected to each other by some edges in tangential direction but the shortest path distance in radial direction forces each block of a fixed ring to be connected to blocks of other rings by at most one block away and this can be done by choosing an appropriate scale of the ring.
The algorithm then chooses a random set(call it $U_{s}(r)$) for all blocks of the solar system. The algorithm chooses randomly until the vertices inside each $U_{s}(r)$ produces a connected subgraph.  
Consider a block at $ring_{s}(r)$, and define the exterior boundary of it as follows:
\begin{equation}\label{eq-exteriorboundary}
B_{\partial}^{(r)}=\{ x\in V: d_{G}(x,B^{(r)})=1 , x\not\in B^{(r)}  \}
\end{equation}  
The probability that a random walk on weighted graph starting from vertex i hits $U_{s}(r)$ before it hits the boundary $B_{\partial}^{(r)}$ is denoted by $q(i,U_{s}(r),B_{\partial}^{(r)}):=\mathbb{P}^{i}[T_U<T_{\partial B}]$
It is known that this type of function is a harmonic function and can be computed efficiently using a linear solver \citep{barlow2017} . Thus, imagine the function q on graph to be the solution of the following linear set of equations
\begin{equation}\label{eq-linearequations}
\begin{split}
\sum_{j}(q(i)-q(j))\frac{w_{ij}}{w_j}&=0 \ \text{if} \ i\in B^{(r)}\setminus U_{s}(r) \\
q(j)&=\begin{cases}
    1, & \text{if} \ j\in U_{s}(r) \\
    0, & \text{if} \ l\in B_{\partial}^{(r)} 
  \end{cases}
\end{split}   
\end{equation} 

\begin{definition}
relative absorption to a set U is a measure of absorption of a random walk to U as follows
\begin{equation}\label{eq-relativeabsorption}
RA=\sum_{i}\frac{q(i)}{d(i,B^{(r)})}  
\end{equation}
\end{definition}
Since $RA$ is calculated for each block, a fixed threshold for all blocks can be set to check if $RA$ is less than the allowed threshold $RA_{al}$.
The algorithm checks if relative absorption to $U_{s}(r)$ is smaller than $RA_{al}$ and in the cases that it is so we call these $U_{s}(r)$, safe sets. One can think that the whole process is like a percolation of light from the center of solar system until it reaches the farest ring in the solar system. So the probability that light passes through some rings and reaches the latest ring can be approximated by a geometric distribution. The approximation is exact if $RA$ is the same for all rings.  
The sets $U_{s}(r)$  are merged to each other if their mutual conductances is higher than a threshold $\phi_{al}$ which controls the sparsity among these sets. After merging many sets together, some giant components appears which we call them long radial constellations and can be considered as a cluster for the output of the algorithm. Apart from these long radial constellations, some clusters are formed by merging safe sets.
Since each solar system covers only a subset of vertices of the graph, the algorithm repeats for many different solar centers and can be done efficiently in parallel. 
\begin{definition}
galaxy is the sum of disjoint solar systems.
\end{definition}
The algorithm then outputs the set of clusters by union of clusters from different solar systems. 

\section{MRP as a dual process}
\begin{definition}
The green function upto ring r  is the green function of the random walk which is killed out of set $A_{r}$ and is denoted by 
\begin{equation}
G_{A_{r}}(x,y)=\sum_{n=0}^{\infty}p^{A_{r}}_{n}(x,y)
\end{equation}
where $A_{r}=\bigcup_{i=1}^{r} Ring(i)$
\end{definition}
\begin{lemma}
The green function upto ring r for a recurrent solar system for center c satisfies
\begin{equation}
G_{A_{r}}(c,c)=\frac{1}{\omega(c)\mathbb{P}^{c}(\tau_{A_r}\leq T^{+}_{c})}
\end{equation}
\end{lemma}
\begin{proof}
Lets define the local time of the Markov chain on c to be
\begin{equation}
L^{c}_{n}=\sum_{r=0}^{n-1}1_{ \{X_{r}=c\} }
\end{equation}
Observe that
\begin{equation}
\mathbb{E}^{x}L^{c}_{\tau_{A_r}}=\sum_{n=0}^{\infty}\mathbb{P}^{x}(X_{n}=c,\tau_{A_r}>n)=g_{A_r}(x,c)\omega_{c}
\end{equation}
Strong Markov property implies that local time on c has a geometric distribution, Thus:
\begin{equation}
\begin{split}
\omega_{c}g_{A_r}(c,c)&=\mathbb{E}^{c}(L^{c}_{\tau_{A_r}}) \\
&=\sum_{k=1}^{\infty}k(\mathbb{P}^{c}(T^{+}_{c}<\tau_{A_r}))^{k-1}(1-\mathbb{P}^{c}(T^{+}_{c}\leq\tau_{A_r})) \\
&=\frac{1}{1-\mathbb{P}^{c}(T^{+}_{c}\leq\tau_{A_r})}  \\
&=\frac{1}{\mathbb{P}^{c}(\tau_{A_r}\leq T^{+}_{c})}
\end{split}
\end{equation}
\end{proof}

The proof of the following two lemma can be found for example in \citep{barlow2017}
\begin{lemma} 
Suppose $X_j$ is the random walk on the solar system,$A_{i}=Ring(i)$,$B_{i}=Ring(i)\cup Ring(i+1)$, $\tau_{i}=\inf\{j\geq 1 : X_{j}\in A_{i} \}$ and $\bar{\tau_{i}}=\inf\{j\geq 0 : X_{j}\in A_{i}  \}$ \\
If $\eta_{i}=\inf\{j\geq 0 :X_{j}\not\in B_{i}\}$ , then 
\begin{equation}
\mathbb{P}^{x}\{ \bar{\tau_{i}}\leq \eta_{i} \}=\sum_{y\in A_{i}}G_{B_{i}}(x,y)\mathbb{P}^{y}\{\tau_{i}> \eta_{i} \}
\end{equation}
holds for all rings $i=1,\ldots,n_r$
\end{lemma}
\begin{proof}
let $\sigma_{i}=\sup\{j: X_{j}\in A_{i}, j\leq \eta_{i} \}$ then
\begin{equation}
\begin{split}
\mathbb{P}^{x}\{\bar{\tau_{i}}\leq \eta_{i}\}&=\sum_{j=0}^{\infty} \mathbb{P}^{x} \{\sigma=j  \}  \\
&=\sum_{y\in A_{i}}  \sum_{j=0}^{\infty} \mathbb{P}^{x} \{\sigma=j  ,X_{j}=y\} 
\end{split}
\end{equation}
Simple markov property produces
\begin{equation}
\begin{split}
&=\sum_{y\in A_{i}}  \sum_{j=0}^{\infty} \mathbb{P}^{x} \{ X_{j}=y;j\leq \eta_{i};X_{k}\not\in A_{i};j<k\leq\eta_{i} \}  \} \\
&=\sum_{y\in A_{i}}  \sum_{j=0}^{\infty} \mathbb{P}^{x} \{ X_{j}=y;j\leq \eta_{i}  \} \mathbb{P}^{y} \{ \tau_{i}>\eta_{i}  \}
\end{split}
\end{equation} 
A simple rearrangement shows that
\begin{equation}
\begin{split}
\mathbb{P}^{x}\{ \bar{\tau_{i}}\leq\eta_{i} \} &= \sum_{y\in A_{i}} \mathbb{P}^{y} \{\tau_{i}>\eta_{i}  \} \sum_{j=0}^{\infty}\mathbb{P}^{x}\{ X_{j}=y;j\leq \eta \}  \\
 &=\sum_{y\in A_{i}}  \mathbb{P}^{y} \{\tau_{i}>\eta_{i} \} G_{B_{i}}(x,y) 
\end{split}
\end{equation}
\end{proof}
\begin{lemma} 
Let $S_n$ be a simple symmetric random walk on $Z$ with $S_{0}=0$ then 
\begin{equation}
\sum_{r\in Z} \lambda^{r}\mathbb{P}(S_{n}=r)=2^{-n}\sum_{r=0}^{n}\lambda^{2r-n}{n \choose r} \ -1\leq t \leq 1
\end{equation}
\end{lemma}

Multiplying chebyshev polynomials $H_{k}(t)=\frac{1}{2}(t+i(1-t^{2})^{\frac{1}{2}})^{k}+\frac{1}{2}(t-i(1-t^{2})^{\frac{1}{2}})^{k}$ by probability thar random walk hits k and summing over all k has the following equality
\begin{lemma}\label{polynomial}
For each k, $H_{k}(t)$ is a real polynomial in t of degree k and $|H_{|k|}(t)|\leq 1$, Further, for each $n\geq 0$ 
\begin{equation}
t^{n}=\sum_{k\in Z} \mathbb{P}(S_{n}=k)H_{|k|}(t)
\end{equation}
\end{lemma}
The importance of Lemma~\ref{polynomial}  is that the scalar can be substituted by the transition probability matrix of the random walk on weighted graph:
\begin{equation}\label{firstdualrelation}
P^{n}=\sum_{k\in Z} \mathbb{P}(S_{n}=k)H_{|k|}(P)
\end{equation}
\eqref{firstdualrelation} shows the duality between the primary process and the random walk on $Z$ which is also used for proof of Carne–Varopoulos bound in Theorem~\ref{thm:carne-var}. Now $p_{n}(x,y)$ can be written as
\begin{equation}
\begin{split}
p_{n}(x,y)&=\langle f_1,P^{n}f_2 \rangle \\
&= \sum_{k\in Z} \mathbb{P}(S_{n}=k)\langle f_{1},  H_{|k|}(P)f_{2} \rangle
\end{split}
\end{equation}
where $f_{i}(x)=\mathds{1}_{x_i}\mu_{x_i}^{-\frac{1}{2}}$ for $i=1,2$
This duality is too strong and it is better to work with some expectations(such as characteristics for real random variable and generating functions for discrete random variables) rather than transition probabilities. In fact in the present paper, convergence in distribution(the weakest form of convergence) is the main tool.
It is well known in the theory of discrete subordination of random walk, any Bernstein function has a Lévy–Khintchine representation. \citep{Schilling2012}
\begin{equation}
\phi(x)=a+bx+\int_{0}^{\infty} (1-e^{xy})\nu(dy), \ x\geq 0
\end{equation}
where $\nu$ is the radon measure.

\section{Many sparse cuts}
As mentioned in the introduction, algorithm~\ref{alg:clustering} is a randomized algorithm that might be practical for cases that involves big data that a parallel algorithm is essential. This algorithm can be combined with evolutionary algorithms to find semi-optimum clusters since each run of algorithm produces a measure such as conductance for the quality of clustering.  
\begin{algorithm}[H]
\caption{outputs many clusters}
\label{alg:clustering}
Input: A weighted undirected graph with n vertices and m edges \\
loop over $n_c$ centers:\\
loop for different randomizations: \\
1: create a metric embedding into $\ell^{2}_{2}$ and save records in polar coordinates inside the database  \\
2: create the associated solar system \\
3: solve the linear equations associated to blocks of each ring in parallel  \\
5: make one random set for each block of each ring \\
4: check if relative absorption to random sets in step 5 is smaller than the threshold $RA_{al}$, if it is smaller, label the random set as safe set.\\
5: merge safe sets if the conductance between them is higher than a threshold($\phi_{al}$) \\   
Output: giant components and merged safe sets of all solar systems 
\end{algorithm}

\section{Entropy}
The concept of optimal transport time or sometimes called access time from $\mu$ to $\nu$ is defined by \citep{Lovasz1995} 
\begin{equation}\label{accesstime}
H(\mu,\nu):=\inf_{\tau\in \Gamma}\mathbb{E}[\tau] 
\end{equation}
where $\Gamma$ is the set of stopping times that transport $\mu$ to $\nu$. This problem has a close relation with skorokhod embedding if one replaces brownian motion with a general discrete markov chain. \citep{Beiglboeck2017} and a practical tool to solve these problems are by PDE and free boundary problem approaches which can be efficiently solved by computers.\citep{Ghoussoub2019}

Now consider the following function
\begin{equation}\label{generalfunction}
V(x)=\sup_{\tau} \mathbb{E}_{x}(M(X_{\tau}))+\int_{0}^{\tau} L(X_{t})dt+ \sup_{0\leq t\leq \tau} K(X_{t})+ F(x)
\end{equation}
Important quantities such as \eqref{accesstime} are just a special case of \eqref{generalfunction} when the functions $L,K,F$ are zero and $M=-\tau$.
Another important special case is average value of a process $Y_{t}=L(X_{t})$ over time and is related to large deviation theory for markov chains and ergodic theory:
\begin{equation}
\begin{split}
A_{t}&=\frac{1}{T}\sum_{i=1}^{T}Y_{i}  \\
Y_{i}&=d(x,X_i)
\end{split}
\end{equation}
where $Y_i$ measures how far the markov chain is from a fixed reference vertex $x$ and the corresponding event can be approximated by rate function $I(a)$ as follows:
\begin{equation}
\mathbb{P}(A_{t}=a)\approx \exp^{-TI(a)}
\end{equation}
One example that is also related to American put option is when \eqref{generalfunction} takes the following special form
\begin{equation}
V(x)=\sup _{\tau} \mathbb{E} e^{-r\tau}(K-X_{\tau})^{+}
\end{equation}
where K is the strike price, r is the interest rate, and it can easily be reduced to free boundary value problem and the corresponding optimal stopping time turns out to exist and is 
\begin{equation}
\tau_{\star}=\inf \{0\leq t \leq T : X_{t}\leq b_{\star}(t)   \}
\end{equation}
Note that function $F(x)$ in \eqref{generalfunction} is independent of time and stopping time. A very important example is when $F(x)$ is a probability that the the discrepancy between shortest path distance between a fixed vertex and a subset $\Omega$ and the time it takes for the random walk to hit $\Omega$,  is C times greater than shortest path distance  as follows
\begin{equation}
F(x)=\mathbb{P}(T(x,\Omega)-d(x,\Omega)> C d(x,\Omega))
\end{equation}
The methodology of algorithm~\ref{alg:clustering2} for clustering is based on including the vertices as much as possible to increase entropy within each cluster.
Two concepts of entropy namely $\alpha$-entropy and $\beta$-entropy is defined. The former is diffucult to compute in practice while the latter is easy to calculate ,it has no limit, and is used in algorithm~\ref{alg:clustering2}. The relation between these two entropies is an open question. Note that at each moment of time, the random walk on the weighted graph belongs to one or more subsets of the state space of the markov chain. Thus an example of trajectroy of the dynamics of a particle going by the random walk could be like $\{\ldots,\Omega_3,\Omega_2,\Omega_2,\Omega_4,\Omega_1,\Omega_2,\ldots \}$ . So at each time the random walk could belong to any of these subsets. For simplicity and without lack of generality assume these subsets are balls of some fixed radious in the shortest path metric on the normalized weighted graph. Thus, $\Omega$ is shown in figure~\ref{fig-covering} and is decomposed as follows:
\begin{equation}\label{decomposition}
\Omega= \cup_{i=1}^{n} \Omega_i
\end{equation}
\begin{figure}[H]
  \includegraphics[width=\linewidth]{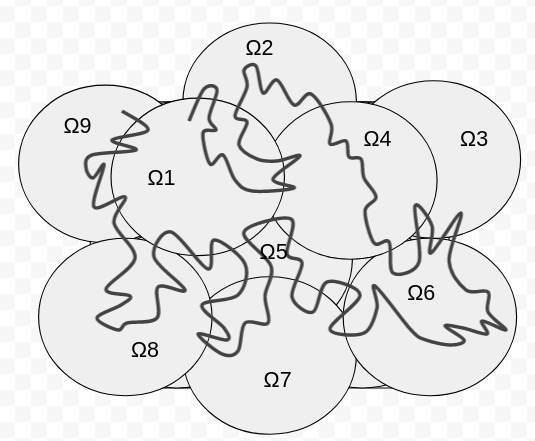}
  \caption{a random walk on weighted graph and the covering $\Omega_i$}
  \label{fig-covering}
\end{figure}

Observer that as n is increased in \eqref{decomposition}, the number of balls are increased to cover all vertices and balls have smaller radius.
\begin{definition}
$\alpha$-entropy is defined as follows
\begin{equation}
H^{(\alpha)}=\lim_{n\to \infty}-\frac{1}{n}\sum_{i_{1}i_{2}\ldots i_{n}}\mathbb{P}(\Omega_{i_1},\Omega_{i_2},\ldots,\Omega_{i_n})\log \mathbb{P}(\Omega_{i_1},\Omega_{i_2},\ldots,\Omega_{i_n})
\end{equation}
\end{definition}

\begin{definition}
$\beta$-entropy is defined as follows
\begin{equation}
H^{(\beta)}=\frac{1}{n}\sum_{i=1}^{n}V_{i}(x)\log(V_{i}(x))
\end{equation}
while functions $V_{i}$ corresponds to \ref{generalfunction} where function $M=d(x,X_{\tau})$ , K is zero, and $L_i$ is associated with $V_i$ and is defined as follows 
\begin{equation}
L_{i}(X_t)=\mathds{1}(X_t\in\Omega_i) , i=1,\ldots,n
\end{equation}
\end{definition}
Define the following stochastic process: \\
\begin{equation}\label{snell-envelope}
\begin{split}
B_{N}^{N}&=M_{N}  \\
B_{n}^{N}&=\max \{M_{n},\mathbb{E}(B_{n+1}^{N}|\mathcal{F}_{n})  \} ,\ n=N-1,N-2,\ldots
\end{split}
\end{equation}
The importance of \eqref{snell-envelope} is that it is the smallest supermartingale that majorizes $M(X_k)$ for $n\leq k \leq N$.
The following class of stopping times is used in the next lemma.
\begin{equation}\label{eq-optimalstoppingtime}
\tau_{n}^{N}=\inf \{k\in \{n,n+1,\ldots,N \}:B_{k}^{N}=M_{k}  \}
\end{equation}
\begin{lemma}\label{majorizingprocess}
Suppose $M=(M_{k})_{k=n}^{N}$ satisfies $\mathbb{E}\max_{n\leq k\leq N} M_{k}<\infty$, then the stopping time defined in  \eqref{eq-optimalstoppingtime} is optimal in \eqref{generalfunction}
\end{lemma}
\begin{proof}
First the proof sketch is described. Observe that it suffices to show that $B_{k}^{N}$ is the smallest supermartingale that majorizes M and $B^{N}_{\tau_{n}^{N}\wedge k}$ is martingale for $n \leq k \leq N$ and then Doob's optional stopping theorem applied to this martingale shows that 
\begin{equation}\label{eq-whp}
\begin{split}
B_{n}^{N}&\geq \mathbb{E}(M_{\tau}| \mathcal{F}_{n})  \\
B_{n}^{N}&=\mathbb{E}(M_{\tau_{n}^{N}}|\mathcal{F}_n)
\end{split}
\end{equation}
Taking expectations of both sides of \eqref{eq-whp} gives \\
\begin{equation}
\mathbb{E}M_{\tau}\leq \mathbb{E}M_{n}^{N}=\mathbb{E}M_{\tau_{n}^{N}}
\end{equation}
which shows that $\tau_{n}^{N}$ is the optimal stopping time.
Now that the proof sketch is given, the pieces of it is proven as follows: The following relation proves that if $\bar B$ is another supermartingale that majorizes $M$ ,it is not the smallest supermartingle.
\begin{equation}
\bar{B}^{N}_{k-1}\geq \max\{M_{k-1},\mathbb{E} \bar{B}_{k}^{N}|\mathcal{F}_{k-1}\} \geq \max\{M_{k-1},\mathbb{E} B_{k}^{N}|\mathcal{F}_{k-1}\}=B_{k-1}^{N}
\end{equation} 
The next relationship proves that $B^{N}_{\tau_{n}^{N}\wedge k}$ is martingale.
\begin{equation}
\mathbb{E}[B^{N}_{\tau_{n}\wedge (k+1)} | \mathcal{F}_k ]=B^{N}_{\tau_{n}^{k}\wedge (k)} \mathds{1}_{ \{ \tau_{n}^{N} \leq k \} } + B^{N}_{k} \mathds{1}_{ \{ \tau_{n}^{N} > k \} }=B^{N}_{\tau_{n}^{N}\wedge k}
\end{equation}
\end{proof}

\begin{algorithm}[H]
\caption{outputs clusters by entropy maximization}
\label{alg:clustering2}
Input: A weighted undirected graph with n vertices and m edges \\
loop: \\
1: generate random sets $K_l$ of vertices of the graph each of fixed cardinality \\
2: calculate $\beta$-entropy for all $K_l$  \\
3: sort entropies in step 2  \\
Output: top high entropy sets in step 3 
\end{algorithm}

\section{Conclusion}
Since the second largest eigenvalue of transition matrix and cheeger inequality characterize the markov chain on graph only globally, the present algorithm is designed for big data problems where computing the eigenvectors and eigenvalues of a big graph is time consuming and not robust to variations inside the graph.

\bibliographystyle{agsm}
\bibliography{hitting}
\end{document}